\documentclass[11pt,a4paper,reqno]{amsart}

\usepackage[margin=2.5cm]{geometry}
\usepackage{amssymb,bm,cancel,color,enumerate,ifthen,graphicx,listings,mathtools,paralist,wrapfig,url,xcolor,xspace}
\usepackage{blkarray}
\usepackage{hyperref}

\newtheorem{theorem}{Theorem}[section]
\newtheorem{corollary}[theorem]{Corollary}
\newtheorem{lemma}[theorem]{Lemma}
\newtheorem{examp}{Example}
\newenvironment{example}{\begin{examp}\small\sf}{\end{examp}}

\newcommand\bc{\begin{center}}
\newcommand\ec{\end{center}}
\newcommand\bn{\begin{enumerate}}
\newcommand\en{\end{enumerate}}
\newcommand\TW{\textwidth}

\makeatletter
\def\rf#1{(\@rf#1,.)}
\def\@rf#1,{\ref{eq:#1}\@ifnextchar . {\@endrf}{, \@rf}}
\def\@endrf.{}
\makeatother

\newcommand\apref[1]{Appendix~\ref{sc:#1}\xspace}

\newcommand\fgref[1]{Figure~\ref{fg:#1}\xspace}
\newcommand\lmref[1]{Lemma~\ref{lm:#1}\xspace}

\newcommand\thref[1]{Theorem~\ref{th:#1}\xspace}
\newcommand\ssrf[1]{\S\ref{ss:#1}\xspace}%abbreviation of \ssref

\renewcommand\_[1]{\mathbf{#1}}
\renewcommand\:{\,{:}\,} %for range like 1:n; ordinary : has too much space

\newcommand\ds{\displaystyle}

\newcommand\ninf{{-\!\infty}}

\newcommand\R{\mathbb{R}}

\renewcommand{\d}{\mathrm{d}}

\newcommand\val{\mathop{\textup{Val}}}
\newcommand\lam{\lambda}

\newcommand\sij{\sigma_{ij}}

\newcommand\set[2]{\{\,#1\mid\mbox{#2}\,\}}
\newcommand\dbd[2]{\frac{\partial #1}{\partial #2}} %displayed generic partial derivative
\newcommand\dbdt[2]{\partial #1/\partial #2} %generic partial derivative for running text

\newcommand\matlab{{\sc Matlab}\xspace}

\newcommand\SA{structural analysis\xspace}
\newcommand\Smethod{$\Sigma$-method\xspace}
\newcommand\sigmx{signature matrix\xspace}

\newcommand\sysJ{system Jacobian\xspace}

\newcommand\nc{\newcommand}
\nc\rnc{\renewcommand}

\nc\bib{Bibliography}
\nc\images{images}
\nc{\workprecdir}{WORK_PREC}

\nc\NN[1]{{\color{red}{Ned: #1}}}
\nc\JP[1]{{\color{blue}{John: #1}}}
\nc\para[1]{\par\medskip\noindent{\bfseries\slshape#1.}}
\nc\HI[1]{{\color{red}#1}}
\nc\HIb[1]{{\color{blue}#1}}
\nc\NC[1]{{\color{red}{Ned: #1}}}

\def\<#1,#2>{#1\cdot#2} %inner product applied-math style
\rnc\_[1]{\mathbf{#1}}
\rnc\`{^{-1}}
\rnc\:{{:}}
\rnc\.[2]{\ifthenelse{#1=1}{\dot{#2}}{\ifthenelse{#1=2}{\ddot{#2}}{\ifthenelse{#1=3}{\dddot{#2}}{??}}}}
\nc\?[1]{\mathsf{\MakeUppercase{#1}}}
\nc{\pnt}[1]{\?{#1}}
\nc\aap{\.1{\_a}}
\rnc\%[1]{\underline{\?{#1}}} %position
\rnc\$[1]{\underline{\_{#1}}} %vector
\nc\Abar{\ol{A}}%{\widehat{A}}
\nc\Ap{\.1{\?a}}
\nc\bbp{\.1{\_b}}
\nc\bbr[1]{\left[\!\left[#1\right]\!\right]}
\nc\Bp{\.1{\?b}}
\nc\C{\mathbb{C}}
\nc\cen{\ol{\?M}} %centroid in local frame
\nc\cenmov{\?M} %centroid moving
\nc\conj[1]{\overline{#1}}
\nc\Cp{\.1{C}}
\nc\cplx{\mathbb{C}}
\nc\cross{{\pmb{\times}}}%uses ``poor man's bold"
\nc\Cross{\pmb{\bigtimes}}
\nc\Ominus{\mathop{\pt{$-$}}}
\nc\Ocross{\mathop{\pt{$\x$}}}
\nc\pt[1]{\text{\textcircled{#1}}}
\nc\ddt{\dfrac{\d}{\d t}}
\nc\DOF{\textsc{dof}\xspace}
\nc\e[1]{\texttt{e#1}}
\nc\KinEn{\textup{KinEn}}
\nc\Uxp{\dot{\Ux}}
\nc\Uyp{\dot{\Uy}}
\nc\Uzp{\dot{\Uz}}
\nc\In{\pmb{\text{I}}}%moment of inertia about centroid
\nc\Inhat{\widehat{\mathbf{I}}_{\?c}}
\nc\KErod[3]{\frac{#1}6 (\.1{\?{#2}}^2 + \.1{\?{#2}}\cdot\.1{\?{#3}} + \.1{\?{#3}}^2)}
\nc\Lag{\mathcal{L}}
\nc\Mmatrix{\_M}%{\mathcal{M}}
\nc\nrm[1]{\left|#1\right|}
\nc\Nrm[1]{\left\|#1\right\|}
\nc\om{\omega}
\nc\omhat{\widehat{\om}}
\nc\phip{\.1{\phi}}
\nc\Parts{\mathcal{P}}
\nc\Qbar{\ol{Q}}
\nc\pp{\.1{p}}
\nc\rp{\.1{r}}
\nc\qp{\.1{\_q}}
\nc\up{\.1{\_u}}
\nc\vp{\.1{\_v}}
\rnc\wp{\.1{\_w}}
\nc\Ux{\widehat{\_x}}
\nc\Uy{\widehat{\_y}}
\nc\Uz{\widehat{\_z}}
\nc\Qec{\ol{P}}
\nc\Qhat{\widehat{Q}}
\nc\Qhatp{\.1\Qhat}
\nc\Qhatec{\widehat{P}}
\nc\Qp{\.1{Q}}
\nc\Rbod{\ensuremath{\mathcal{R}}\xspace}
\nc\Rec{\ol{R}}
\nc\thetap{\.1{\theta}}
\nc\thOA{\theta_{\?O\?A}}
\nc\rrp{\.1{\_r}}
\nc\xp{\.1{x}}
\nc\yp{\.1{y}}
\nc\zp{\.1{z}}
\nc\uhat{\widehat{u}}
\rnc\vec[1]{\overline{#1}}
\nc\x{\times}
\nc\X{\ensuremath{\mathcal{X}}\xspace}
\nc\xhat{\widehat{x}}
\nc\Xbar{\ol{X}}
\nc\xibar{\ol{\xi}}
\nc\rfsol{Q}
\nc\tol{\text{\tt tol}}
\nc\csol{Q^{\tol}}
\nc\tend{t_\text{end}}
\nc\scd{\text{\sc scd}}
\nc\m{\text{m}}
\nc\s{\text{s}}
\nc\Kg{\text{Kg}}

\nc\inter{\shortintertext}
\nc\mx[1]{\begin{matrix}#1\end{matrix}}
\nc\bmx[1]{\begin{bmatrix}#1\end{bmatrix}}
\nc\pmx[1]{\begin{pmatrix}#1\end{pmatrix}}
\nc\smallmx[1]{\begin{smallmatrix}#1\end{smallmatrix}}
\nc\smallbmx[1]{\left[\begin{smallmatrix}#1\end{smallmatrix}\right]}
\nc\ypos{\textup{YPos}}
\nc\ol[1]{\overline{#1}}

\nc\cpp{\texttt{C++}\xspace}
\nc\daets{\textsc{Daets}\xspace}
\nc\klu{\textsc{Klu}\xspace}
\nc\ipopt{\textsc{Ipopt}\xspace}
\nc\dof{degrees of freedom\xspace}
\nc\fadbad{\textsc{FADBAD++}\xspace}
\nc\KE{kinetic energy\xspace}
\nc\PE{potential energy\xspace}
\nc\lagfac{Lagrangian facility\xspace}
\nc\mechfac{mechanism facility\xspace}
\nc\MoI{moment of inertia\xspace}
\nc\NCs{Natural Coordinates\xspace}
\nc\ncs{natural coordinates\xspace}
\nc\UP{useful point\xspace}
\nc\UPs{\UP{s}\xspace}
\nc\BP{basic point\xspace}
\nc\BPs{\BP{s}\xspace}
\nc\BPV{basic point or vector\xspace}
\nc\BPVs{basic points or vectors\xspace}
\nc\TS{Taylor series\xspace}
\nc\WF{world-frame\xspace}
\nc\wrt{with respect to\xspace}
\nc{\yml}{{\sc Yaml}\xspace}
\nc\JB{Jalon1994KDS}
\nc\KWB{krauswincklerbock2001}
\nc\vS{vonschwerin2012MBS}

\nc\lnref[1]{line~\ref{ln:#1}}
\nc\lnsref[2]{lines~\ref{ln:#1}--\ref{ln:#2}}
\nc\Lnref[1]{Line~\ref{ln:#1}}
\nc\Lnsref[2]{Lines~\ref{ln:#1}--\ref{ln:#2}}

\nc\li{\protect\lstinline}

\nc\lii{\protect\lstinline[keywords={}]}
\nc\pH{port-Hamiltonian\xspace}
\nc\Sch{\cite{Scholz2017sigmacircuits}\xspace}
\nc\Schref[1]{\cite[#1]{Scholz2017sigmacircuits}\xspace}

\nc\ddbd[3]{\frac{\partial^2#1}{\partial#2\partial#3}} %2nd partial derivative
\nc\nT{n{-}1}
\nc\nN{m{-}n{+}1}
\rnc\^[1]{\widehat{#1}}
\rnc\?[1]{\pmb{\mathcal{#1}}}
\nc\calC{\mathcal{C}}
\nc\calG{\mathcal{G}}
\nc\calI{\mathcal{I}}
\nc\calL{\mathcal{L}}
\nc\calR{\mathcal{R}}
\nc\calT{\mathcal{T}}
\nc\calV{\mathcal{V}}
\nc\F{\widetilde{F}}
\rnc\v{\pmb{\nu}}
\rnc\i{\pmb{\iota}}
\nc\I{\mathbb{I}}

\rnc\$[2]{F_{#1#2}} %for the F matrices
\rnc\%[2]{\^F_{#1#2}} %ditto
\rnc\s[1]{\mbox{\footnotesize\ensuremath{#1}}}
\nc\st[1]{\mbox{\footnotesize #1}}
\nc\twofbox[6]{{\footnotesize
\begin{blockarray}{rcc>{\scriptsize}l}
&#1  &c_i\\     %e.g.  {x1&x2}
\cline{2-3}
\begin{block}{r|cc|l}
#2&&&#3 \\ %e.g.  {f1} {m1}
#4&&&#5 \\ %e.g.  {f2} {m2}
\cline{2-3}
\end{block}
d_j&#6        %e.g.  {n1&n2}
\end{blockarray}
}}

\begin{document}

\title[Compact port-Hamiltonian circuit analysis is SA-amenable]{A compact, structural analysis amenable, port-Hamiltonian circuit analysis}

\author{John D.~Pryce}

\begin{abstract}
This article presents a simple port-Hamiltonian formulation of the equations for an RLC electric circuit as a differential-algebraic equation system, and a proof that structural analysis always succeeds on it for a well-posed circuit, thus providing a correct regularisation for numerical solution.
The DAE is small: its size is at most the number of edges in the circuit graph.
\end{abstract}

\maketitle

%%%%%%%%%%%%%%%%%%%%%%%%%%%%%%%%%%%%%%%%%%%%%%%%%
\section{Introduction}\label{sc:intro}

\subsection{Overview}
Presented here is a simple \pH (pH) formulation of electric circuit equations as a DAE and a proof that \SA (SA) always succeeds on it for a well-posed circuit, thus providing a correct regularisation for numerical solution.
We call it {\em Compact port-Hamiltonian} analysis (CpH).

CpH can be stated in few words.
{\em Choose the pH-recommended (e.g.\ \cite[Appendix B.1]{vanderschaftjeltsema2014}) state variable for each element, giving a state vector with a component for each of the $m$ edges of the circuit graph; express the edge voltages and currents as explicit functions of the state; using an optimal tree according to Tischendorf index theory, apply Kirchoff's current law KCL across the cutset defined by each tree edge, and voltage law KVL round the loop defined by each cotree edge.}
\smallskip

The resulting DAE has size $m$; in fact less, if one exploits the fact that state variables associated with voltage and current sources can be removed and treated as output variables.

Our preferred SA method is the \Smethod \cite{Pryce2001a}. If, as here, only first derivatives occur in the equations then the Pantelides method \cite{Pant88b}, the \Smethod and the Mattsson--S\"oderlind dummy derivatives method \cite{Matt93a} are equivalent in the sense that if one succeeds, they all do.
\smallskip

In a pH approach, see e.g.\ \cite{vanderschaftjeltsema2014},
\begin{enumerate}[(1)]
  \item A system is  a union of {\em energy-storing}, {\em energy-dissipating} and {\em energy-routing} elements connected by {\em ports}.
  \item The energy-storing part is represented by a Hamiltonian $H$ giving the system's total energy. State variables are chosen so that $H$ is an algebraic function of them.
  \item Energy-routing is specified by a {\em Dirac structure}; for circuits it is the incidence matrix of the circuit graph, or an equivalent mathematical object, that describes the topology.
\end{enumerate}
\smallskip

We make much use of methods in the 2017 Lena Scholz report \Sch referred to for short as LS.
Her overview of the \Smethod in LS \S4.1 suffices here and it will not be described further.

But a seemingly minor difference between the LS approach and ours has a deep effect.
Throughout \Sch including in the \pH sections she chooses that
\begin{compactitem}
  \item for a capacitor, charge $q$ is a function of voltage;
  \item for an inductor, flux $\phi$ is a function of current.
\end{compactitem}
By contrast the recommended pH approach takes $q$ and $\phi$ as state variables and thus {\em independent}.

This reversal is what simplifies CpH.
The Hamiltonian $H$ is a sum of terms including $H_C(q) =$ capacitor energy and $H_L(\phi) =$ inductor energy.
Capacitor voltage $\v$ appears as $\dbdt{H}{q}$, current as $\dot{q}$. Inductor current $\i$ appears as $\dbdt{H}{\phi}$, voltage as $\dot{\phi}$.
E.g.\ a capacitor with linear behaviour contributes $q^2/2C$ to $H$ so $\v=q/C$.

LS states five assumptions (A1)--(A5) underlying her analysis.
(A1)--(A2) are needed because the circuit elements are ideal. They forbid%
\footnote{But joining, e.g., two real car batteries in this way is OK---precisely because they are {\em not} ideal sources.},
e.g., two voltage sources $V(t),V'(t)$ with the same start and end nodes because this makes the circuit contradictory (if $V(t)\ne V'(t)$ can happen) or underdetermined (if it can't).

(A3)--(A5) state {\em passivity}---non-source elements can't create energy from nothing.
Implicit in their wording is that charge on each capacitor might depend jointly on all capacitor voltages but on nothing else; flux on each inductor might depend jointly on all inductor currents but nothing else; current in each resistor might depend jointly on all resistor voltages but nothing else.

We restate (A3) for capacitors and (A4) for inductors in terms of $H$, keeping the same ``jointly'' behaviour.
The technical condition is that the Jacobians of the above dependences be symmetric positive definite (SPD).
Dependence reversal and the relation to the Hamiltonian imply
\begin{align}
  \dbdt{q}{\v} = \left(\dbdt{^2 H}{q^2}\right)\`,\qquad
  \dbdt{\phi}{\i} = \left(\dbdt{^2 H}{\phi^2}\right)\`.
\end{align}
A matrix is SPD iff its inverse is, so (A3), (A4) in our context require these Hessian matrices of $H$ to be SPD.
However we use (A5) for resistors unaltered.

\subsection{The LS assumptions rewritten}\label{ss:LSassumptions}
We take the content of the assumptions in LS \S3.1 unchanged, but reword them to match the pH approach, and because our incidence matrix is the transpose of hers.

We take the incidence matrix $A$ of the circuit graph $G$, with $m$ edges and $n$ nodes, to be $m\x n$, so rows mean edges. In each row $a_{ij}$ is $+1$ if $j$ is the start node of the $i$th edge, $-1$ if the end node, and $0$ otherwise.
Assume $G$ is connected, so $A$ has rank $\nT$ and its null space $\set{x\in\R^n}{$Ax=0$}$ is the span of $e = (1,\ldots,1)^*$.
(Because ``$T$'' is ubiquitous meaning a tree, $^*$ instead of $^T$ is used for transpose throughout.)

$B$ is the reduced incidence matrix, $A$ with one column removed corresponding to a grounded node. 
The rows are assumed (re-)ordered by type of component so $B$ is split into blocks $[B_{\?C};\ B_{\?L};\ B_{\?R};\ B_{\?V};\ B_{\?I}]$ where $B_{\?C}$ describes the capacitors, etc. Here semicolon denotes vertical catenation as in \matlab.

{\em
We say that the DAE system (6) is well-posed if it satisfies the following assumptions.
\begin{enumerate}[\bf({A}1)\hspace{-.3em}]
\item The circuit contains no $\?V$-loops, i.e., $B_{\?V}$ has full row rank.
\item The circuit contains no $\?I$-cutsets, i.e., $[B_{\?C}\ B_{\?L}\ B_{\?R}\ B_{\?V}]$ has full column rank.
\item$'$ The Hamiltonian is a twice continuously differentiable function of the vector of charge $q$ on the capacitors, and the Hessian
  \[ \?C(q) := \dbd{^2 H}{q^2} \]
is pointwise SPD.
\item$'$ The Hamiltonian is a twice continuously differentiable function of the vector of flux $\phi$ on the inductors, and the Hessian
  \[ \?L(\phi) := \dbd{^2 H}{\phi^2} \]
is pointwise SPD.
\item The conductance function $g$ is a continuously differentiable function of the vector of voltage across the resistors, and the Jacobian
  \[ \?G(\v) := \dbd{g}{\v} \]
is pointwise SPD.
\end{enumerate}
}
Owing to dependence reversal, our $\?C$ and $\?L$ are the {\em inverse} of those in LS.

\section{The CpH equations}

\subsection{DAE details}

We expand the short description given at the start.
Equations (11a, 11b) in LS \S3.3 constitute the CpH equations.
We write them as
\begin{align}\label{eq:CpH1}
  \v_N = -F \v_T, \qquad \i_T = F^* \i_N,
\end{align}
forming the model's Dirac structure in the pH sense.
We call $F$ the loop-cutset matrix, see \ssrf{Fmatrix}.
In LS they are part of the {\em Branch-Oriented Model} BOM, shown in LS \S4.4 to be SA-amenable. Though its DAE has size $2m$, not $m$, it is clear BOM and CpH are closely related; much of our argument adapts Scholz's BOM analysis.

In \rf{CpH1}, $\v,\i$ are the vectors of edge voltages and currents and have been split according to some tree $T$ of $G$ into
\begin{align}
  \begin{aligned}
  \v_T, \i_T &= \text{vectors of twig voltages/currents, of length $\nT$}, \\
  \v_N, \i_N &= \text{vectors of link voltages/currents, of length $\nN$}.
  \end{aligned}
\end{align}
We identify the edges with indices in $1\:m$ according to some chosen numbering, and regard $T$ as a subset of $1\:m$, of size $\nT$.
The cotree is the complement $N=1\:m \setminus T$, of size $\nN$ (where $N$ is for ``non-tree'').

As in LS, tree edges are {\em twigs}, cotree edges are {\em links}.
Each twig $f$ defines a (fundamental) {\em cutset}: the set of links whose addition reconnects $T\setminus\{f\}$.
Each link $e$ defines a (fundamental) {\em loop}: the set of twigs whose removal does not disconnect $T\cup\{e\}$.

\para{CpH DAE}
\begin{enumerate}[1.]
  \item For each edge $i=1\:m$ define state-variable $x_i$ according to the type of edge element:
  \begin{align}\label{eq:statevardef}
  \begin{tabular}{c|l|l}
    \bf element code &\bf element type &\bf state variable  \\\hline
     $C$ & capacitor     & charge $q$ \\
     $L$ &inductor       &flux linkage $\phi$ \\
     $V$ &voltage source &current $\i$ through it \\
     $I$ &current source &voltage $\v$ across it \\
     $R$ &resistor       &voltage $\v$ across it
  \end{tabular}
  \end{align}
  This specifies the $m$-vector $x$ of state variables.
  \item Use the constitutive equations---details in equation \rf{videf1}---to express the $m$-vectors $\v,\i$ of edge voltages and currents as functions of $x$, $\xp$ and (for a source) $t$.\\
  Assumptions (A3)--(A5) support both nonlinear element behaviour, and limited dependence between elements.
  \item For each twig of $T$, apply KCL across the cutset it defines; for each link, apply KVL round the loop it defines.
  This gives $\nT$ KCL equations in terms of $\i=\i(t,x,\xp)$, and $\nN$ KVL equations in terms of $\v=\v(t,x,\xp)$.
  
  These equations, which form $f(t,x,\xp)=0$, are precisely \rf{CpH1}.
\end{enumerate}
The $m$ equations used depend upon the chosen tree $T$, but the variables are always the same.

\subsection{Computing the loop-cutset matrix}\label{ss:Fmatrix}

The matrix $F$ can be found as follows. $T$ is a tree iff the corresponding $\nT$ rows of $A$ are a basis of $A$'s row space.
By elementary column operations, $A$ can be transformed to reduced column echelon form, say $\F$, where column $n$ is zero and the tree-rows become the $\nT$ unit vectors $(1,0,\ldots,0),\;(0,1,\ldots,0),\;\ldots$ in some chosen order.
Then $F$ comprises the cotree-rows of $\F$, less the last column.
Think of the row numbers $1\:m$ as permanent labels attached to the edges. 
By a permutation matrix $P$ we can put $\F$ in the form 
\begin{align}
  P \F &=
\begin{blockarray}{rccl}
              &~ \\
\begin{block}{r[c|c]l}
  \st{tree}   &\I &0 &\s{\nT} \\ \cline{2-3}
  \st{cotree} &F   &0 &\s{\nN} \\
\end{block}
              &\s{\nT} &\s1 & \\
\end{blockarray}
\quad= P A X \quad\text{for some nonsingular $X$}
\end{align}\\[-4ex]
where the labels are permuted with the rows.
Now $\I$, the size $\nT$ identity, carries the labelling of its rows to a labelling of the $\nT$ columns of $F$, so discarding these top $\nT$ rows we have
\begin{align}
\raisebox{-1ex}{\begin{blockarray}{rcl}
              &\st{tree $T$} \\
\begin{block}{r[c]l}
  \st{cotree $N$} &$F$         &\s{\nN} \\
\end{block}
              &\s{\nT} & \\
\end{blockarray}
}
= Q A Y
\end{align}\\[-4ex]
where $Q$ is $P$ less its first $\nT$ rows and $Y$ is $X$ less its last column. In \matlab, all this can be achieved by the single statement
\begin{align}\label{eq:computeF}
 \texttt{F = -A(N,1:end-1)/A(T,1:end-1);}
\end{align}
where integer arrays \li{T,N} enumerate $T$ and $N$ in some chosen order.
The minus sign is inserted to accord with the sign convention in \rf{CpH1}, which is used in LS.

Like $A$, matrix $F$ has only entries in $\{-1,0,1\}$. 
Use the notation $F[e,f]$ to mean the entry of $F$ in the row and column labelled by link $e$ and twig $f$ respectively.
$F$ encodes the fundamental cutsets and loops of $T$, namely
\begin{align}
  F[e,f]\ne0 &\iff \text{($f$ is in the loop defined by $e$)} \iff \text{($e$ is in the cutset defined by $f$)},
\end{align}
with the sign, $\pm1$, giving the orientation of the edge within the loop or cutset.
It is now easily seen that \rf{CpH1} defines exactly the equations in step 3 of the CpH model.

\begin{example}
\begin{figure}
\includegraphics[width=.5\TW]{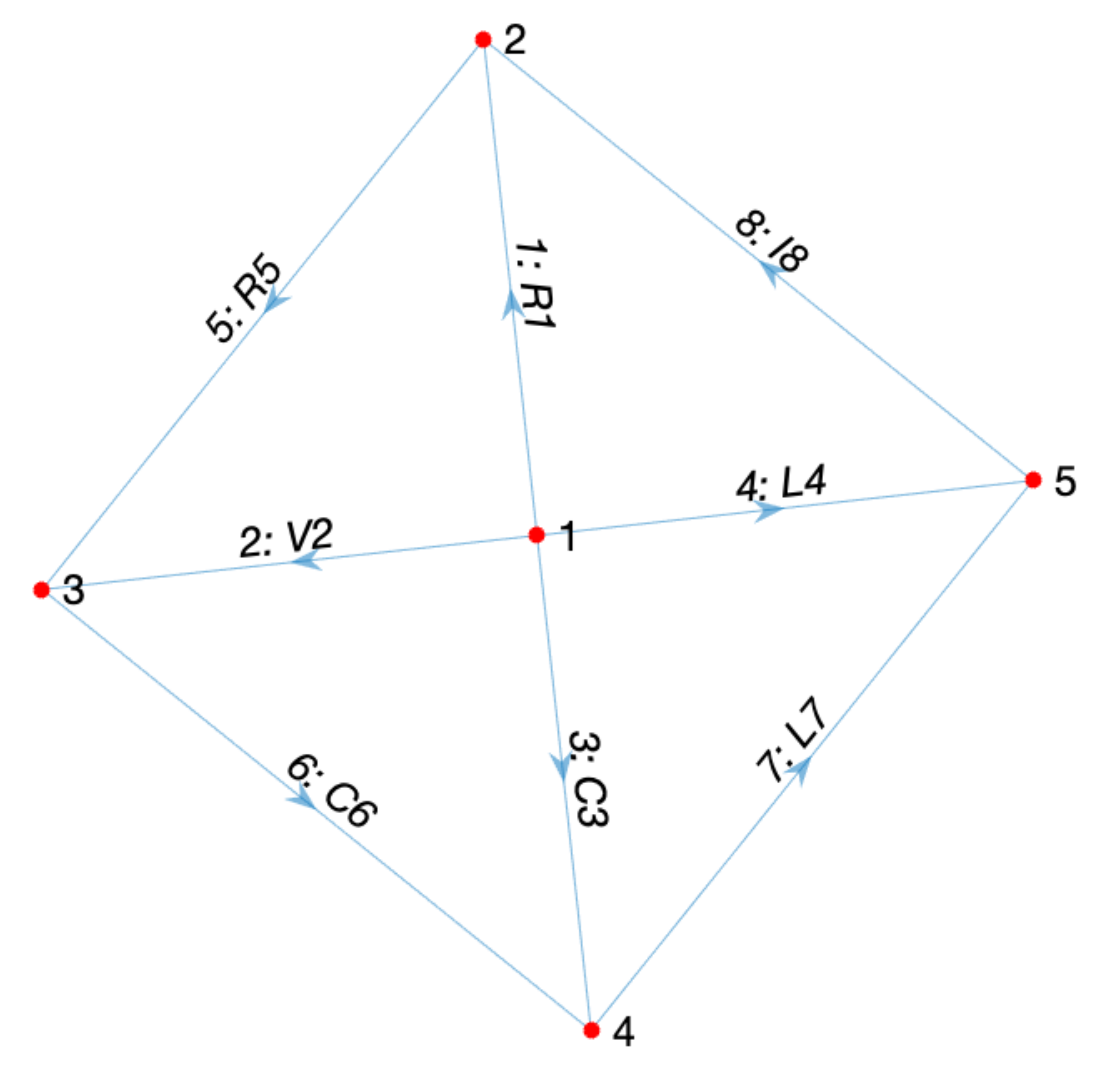}
\caption{8-edge, 5-node example. \label{fg:CpH8by5ex}}
\end{figure}
As a running example we use the 8-edge, 5-node circuit in \fgref{CpH8by5ex}.
We choose as optimal tree (there are several) edges $1,2,3,4$ joined to node $1$ at the centre.
Thus the basic data are
\begin{align}
  &A= \begin{blockarray}{r[rrrrr]}
  \s{r1} & 1 &-1 & 0 & 0 & 0 \\
  \s{v2} & 1 & 0 &-1 & 0 & 0 \\
  \s{c3} & 1 & 0 & 0 &-1 & 0 \\
  \s{l4} & 1 & 0 & 0 & 0 &-1 \\
  \s{R5} & 0 & 1 &-1 & 0 & 0 \\
  \s{C6} & 0 & 0 & 1 &-1 & 0 \\
  \s{L7} & 0 & 0 & 0 & 1 &-1 \\
  \s{I8} & 0 &-1 & 0 & 0 & 1
  \end{blockarray}\ ;
  &&\mx{T=\{1,2,3,4\},\\ N=\{5,6,7,8\};}
  &F=& 
\begin{blockarray}{rrrrr}
          &\s{r1} &\s{v2} &\s{c3} &\s{l4} \\
  \begin{block}{r[rrrr]}
   \s{R5} &  1 & -1 &  0 &  0 \\
   \s{C6} &  0 &  1 & -1 &  0 \\
   \s{L7} &  0 &  0 &  1 & -1 \\
   \s{I8} & -1 &  0 &  0 &  1 \\
  \end{block}
\end{blockarray}\ .
\end{align}\vskip-3ex\noindent
where $F$ was computed from $A$ by \rf{computeF}.
Rows of $A$, and rows and columns of $F$, are labelled by the name of the circuit element on the edge, which includes the edge number. Tree items are lowercase, cotree items are uppercase, to match the notation later.

Much reordering of variables and equations happens as the example progresses---this is just to illustrate the proof.
\apref{excode} shows \cpp code for this example, which is generated without reordering.
Its state vector $x$ is in the order of the circuit edge numbers so following the recipe in \rf{statevardef} we write
\begin{align}\label{eq:exstatevar}
  x &= [x_1,\ldots,x_8]^* = 
  \bmx{\v_{r1},& \i_{v2},& q_{c3},& \phi_{l4},& \v_{R5},& q_{C6},& \phi_{L7},& \v_{I8}}^* ,
\end{align}
where the subscript is a mnemonic for both the numeric index and the element type.

\end{example}

\subsection{Where did node potentials go?}

Common circuit models such as MNA use a grounded node at zero potential and a vector $\lam$ of node potentials at the other $\nT$ nodes, which act as Lagrange multipliers and appear as part of the vector of unknowns, now $(x,\lam)$, of the DAE.
Its Kirchoff equations are $\v + B\lam = 0$ and $B^*\i = 0$ where $B$ is the reduced incidence matrix of \ssrf{LSassumptions}.

A constant nonsingular linear transformation $\lam = Z\mu$ gives equivalent Kirchoff equations $\v + C\mu = 0$ and $C^*\i = 0$ where $C=BZ$ and $\mu$ is a vector of ``generalised potentials''.
CpH may be viewed as the transformed model that chooses $Z$ so {\em $\mu$ becomes the tree voltages} which---for a pH approach---are explicit functions of $t,x,\xp$ and needn't go in the vector of unknowns.
Either form can be converted to the other, so the DAEs of CpH and MNA are equivalent.

\section{Main results}\label{sc:mainthm}

\subsection{Statement of main theorem}

The choice of tree $T$ determines how easy it is to regularise the DAE for numerical solution.
We make a sixth assumption as in \Sch, that $T$, e.g.\ via an algorithm of Tischendorf \cite{Tischendorf99topologicalindex}, is optimal in the sense
\begin{enumerate}[\bf({A}1)]
  \setcounter{enumi}{5}
  \item $T$ has the {\em most possible CV elements} and {\em fewest possible LI elements}.
\end{enumerate}

Our aim is to prove
\begin{theorem}\label{th:main}
  If (A1)--(A5) of \ssrf{LSassumptions} and (A6) hold, then the CpH DAE is SA-amenable.
\end{theorem}
In terms of the \Smethod this means \SA finds a HVT giving the DAE's correct number of degrees of freedom (DOF); the structural index $\nu_s$ equals the differentiation index $\nu_d$; and we know how many times to differentiate each equation to reduce the DAE to an ODE.
The proof is in the following lemmas.
It mainly consists in finding offset vectors $c$, $d$ that lead to a nonsingular \sysJ, and showing these are in fact the canonical offsets.

\subsection{DAE in block form}

Property (A6) plus well-posedness assumptions (A1,2) imply that V elements occur {\em only} in $T$, and I elements occur {\em only} in $N$.
Hence one can sub-partition $T$ and $N$ into disjoint index sets within $1\:m$ thus:
\begin{align}\label{eq:8kinds}
\begin{aligned}
  T &= c\cup l\cup v\cup r,\\
  N &= L\cup C\cup I\cup R,
\end{aligned}
\end{align}
where set $c$ indexes the capacitors on $T$, and so on.
Referring back to the notation in the Assumptions, the names $\?C,\?L,\?V,\?I,\?R$ double up as index sets, a partition of $1\:m$. Then $c$ and $C$ are the intersections of $\?C$ with $T$ and $N$, and so on.

Matrix $F$ (modulo suitable re-ordering of the edges) then takes the block form
\begin{align}\label{eq:blockF1}
  F &=\
  \begin{array}{c@{}c}
    &\text{tree} \\
  \rotatebox[origin=c]{90}{cotree} &
  \begin{blockarray}{rccccl}
    &    c  &   l &    v &    r \\
  \begin{block}{r[cccc]l}
  L & \$Lc & \$Ll & \$Lv & \$Lr \\
  C & \$Cc &   0  & \$Cv &   0  \\
  I & \$Ic & \$Il & \$Iv & \$Ir \\
  R & \$Rc &   0  & \$Rv & \$Rr \\
  \end{block}
  \end{blockarray}
  \end{array}
\end{align}
As explained in LS (\S3.3 above eqn (13)), the blocks in positions $Cl, Cr, Rl$ must be zero because any nonzero in them lets an inductor in $T$ be swapped out or a capacitor outside $T$ swapped in, contradicting the optimality (A6).
In the Kirchoff equations write, following this partition,
\begin{align}\label{eq:vecptn1}
  \v_T &= \bmx{\v_c\\ \v_l\\ \v_v\\ \v_r} &\v_N &= \bmx{\v_L\\ \v_C\\ \v_I\\ \v_R},
    &\text{$\i_T$, $\i_N$ similarly, and }
  &&x &= \bmx{x_T\\x_N}
    = \bmx{x_c\\ x_l\\ x_v\\ x_r\\\hline x_L\\ x_C\\ x_I\\ x_R}
    = \bmx{q_c\\ \phi_l\\ \i_v\\ \v_r\\\hline \phi_L\\ q_C\\ \v_I\\ \v_R}.
\end{align}
The last column is from the pH definition of the state vector $x$, where $q_c$ is the vector of charges on the tree capacitors, $\phi_l$ is fluxes on the tree inductors, etc.

By the constitutive equations, the voltage and current vectors become functions of $x$ and $\xp$:
{\rnc\arraystretch{1.3}
\begin{align}\label{eq:videf1}
  \v &= \bmx{\v_T\\ \v_N} =
    \bmx{\dbd{H}{x_c}(x_c,x_C)\\ \xp_l\\ V(t)\\ x_r\\\hline
      \xp_L\\ \dbd{H}{x_C}(x_c,x_C)\\ x_I\\ x_R},
  &\i &= \bmx{\i_T\\ \i_N} =
    \bmx{\xp_c\\ \dbd{H}{x_l}(x_l,x_L)\\ x_v\\ g_T(x_r,x_R)\\\hline
      \dbd{H}{x_L}(x_l,x_L)\\ \xp_C\\ I(t)\\ g_N(x_r,x_R)},
\end{align}
}%
where $H$ is the Hamiltonian.
The dependences ``$(x_c,x_C)$'' etc.\ express what is implicit in assumptions (A3)--(A5): as groups, the behaviour of capacitors, inductors and resistors might depend on the state of other elements in the same group but not on other groups.
%The analysis must change if there is more general dependence.
%In the simple case of, say, capacitors that are independent and linear, they contribute to $H$ a sum of terms $\frac{q^2}{2C}$ so the relevant $H$ derivatives just give entries $\v=q/C$ for each capacitor; similarly for inductors and resistors.

\begin{example}
In the running example, each of the categories in \rf{8kinds} has just one member:
\begin{align}
  \mx{ c & l & v & r & L & C & I & R \\\hline
      c3 &l4 &v2 &r1 &L7 &C6 &I8 &R5 }
\end{align}
The tree $T=\{1,2,3,4\}$ is optimal since the CV element outside it, $C6$, can't be swapped in without pushing out another CV element, and the LI element inside it, $L7$, can't be swapped out without pushing in another LI element.
Below is the state vector $x$ in the order of \rf{vecptn1}, and the matrix $F$ reordered to match:
\begin{align}\label{eq:vecptn2}
  &&x &= \bmx{x_T\\x_N} \quad\text{with}\quad
    x_T= \bmx{q_{c3}\\ \phi_{l4}\\ \i_{v2}\\ \v_{r1}},\quad
    x_N= \bmx{\phi_{L7}\\ q_{C6}\\ \v_{I8}\\ \v_{R5}},
%\end{align}
%\begin{align}\label{eq:blockF2}
  &F &=\
  \begin{blockarray}{rrrrr}
            &\s{c3}&\s{l4}&\s{v2}&\s{r1} \\ \begin{block}{r[rrrr]}
     \s{L7} &  1   & -1   &  0   &  0    \\
     \s{C6} & -1   &  0   &  1   &  0    \\
     \s{I8} &  0   &  1   &  0   & -1    \\
     \s{R5} &  0   &  0   & -1   &  1    \\
    \end{block}
  \end{blockarray}.
\end{align}
Note $F$ indeed has zeros where \rf{blockF1} says it should.
\end{example}

The DAE $f=0$ consists of \rf{CpH1} after inserting \rf{blockF1,videf1}.
Each equation of \rf{CpH1} belongs to a cutset if it is an $\i$ equation, hence to a unique twig; or a loop if a $\v$ equation, hence to a unique link.
In block form it is
\begin{align}\label{eq:DAEinorder}
  0 = f &= \bmx{f_T\\f_N} = \bmx{\i_T - F^*\i_N\\ \v_N + F\v_T}
\end{align}
where
{\rnc\arraystretch{1.2}
\begin{align}\label{eq:f_Tf_N2}
\begin{aligned}
  f_T = \bmx{f_c\\ f_l\\ f_v\\ f_r} &= \bmx{\xp_c\\ \dbd{H}{x_l}(x_l,x_L)\\ x_v\\ g_T(x_r,x_R)}
    &-& \bmx{\$Lc^* &\$Cc^* &\$Ic^* &\$Rc^*\\  \$Ll^* &0 &\$Il^* &0\\
             \$Lv^* &\$Cv^* &\$Iv^* &\$Rv^*\\  \$Lr^* &0 &\$Ir^* &\$Rr^*}
       \bmx{\dbd{H}{x_L}(x_l,x_L)\\ \xp_C\\ I(t)\\ g_N(x_r,x_R)} , \\[.5ex]
  f_N = \bmx{f_L\\ f_C\\ f_I\\ f_R} &= \bmx{\xp_L\\ \dbd{H}{x_C}(x_c,x_C)\\ x_I\\ x_R}
    &+& \bmx{\$Lc &\$Ll &\$Lv &\$Lr\\  \$Cc &0 &\$Cv &0\\
             \$Ic &\$Il &\$Iv &\$Ir\\  \$Rc &0 &\$Rv &\$Rr}
       \bmx{\dbd{H}{x_c}(x_c,x_C)\\ \xp_l\\ V(t)\\ x_r} .
\end{aligned}
\end{align}
}

Here $x_v$ and $x_I$, belonging to V and I elements, are output variables explicitly given by the $f_v$ and $f_I$ equations, occurring nowhere else. 
We can remove them from the DAE and compute them afterward, e.g.\ $f_v$ says 
$x_v =\$Lv^* \dbdt{H}{x_L}(x_l,x_L) + \$Cv^* \xp_C + \$Iv^* I(t) + \$Rv^* g_N(x_r,x_R)$.

\begin{example}
In the example, assume the components are independent and linear.
Let their names, written with a subscript, double up as their numeric value.
So for instance $c3$ contributes to $H$ a term $q_{c3}^2/2c_3$ so the relevant $H$ derivative gives an entry $\v=q/c_3$.
Thus
{\rnc\arraystretch{1.4}
\begin{align}\label{eq:videf2}
  \v &{=} \bmx{\v_T\\ \v_N},\; \i {=} \bmx{\i_T\\ \i_N}\qquad \text{with}\
  \v_T{=} \bmx{\dfrac{q_{c3}}{c_3}\\ \dot\phi_{l4}\\ v_2(t)\\ \v_{r1}},\
  \v_N{=} \bmx{\dot\phi_{L7}\\ \dfrac{q_{C6}}{C_6}\\ \v_{I8}\\ \v_{R5}};\quad
  \i_T{=} \bmx{\dot{q}_{c3}\\ \dfrac{\phi_{l4}}{l_4}\\ \i_{v2}\\ \dfrac{\v_{r1}}{r_1} },\
  \i_N{=} \bmx{\dfrac{\phi_{L7}}{L_7}\\ \dot{q}_{C6}\\ I_8(t)\\ \dfrac{\v_{R5}}{R_5} },
\end{align}
}%
and the DAE, in terms of $x$ in \rf{vecptn2}, consists of \rf{CpH1} after inserting \rf{videf2} and the reordered $F$ in \rf{vecptn2}.
\end{example}

\subsection{Signature matrix}

From \rf{f_Tf_N2} we form a schematic \sigmx $\Sigma$.
Clearly all entries $\sij$ are $\ninf$, $0$ or $1$.
The $v,I$ rows and columns have been removed, and rows and columns reordered with items of the same type together.
The $n_C$, etc.\ are block sizes. 
Provisional values of the offsets $c_i, d_j$ are added---part of the proof is to show these are in fact the canonical offsets.
Their meaning is blockwise, e.g.\ $c_i=1$ everywhere in the $f_C$ and $f_l$ rows and $c_i=0$ elsewhere.
{\rnc\![1]{\colorbox{lightgray}{$#1$}}
\begin{align}\label{eq:Sigma}
  \Sigma \ &=\
  \begin{blockarray}{rcccccclc}
        &    x_C    &    x_c     &  x_l      &   x_L      &  x_r    &   x_R   &\s{c_i}&     \\
    \begin{block}{r[cc|cc|cc]lc}
    f_C &\!{\le0}   &\!{\le0}    &  -        &  -         &    -    &    -    &\s1    & n_C \\
    f_c &\!{\le1}   &\!{\^\I{+}1}&  \le0     & \le0       &  \le0   &  \le0   &\s0    & n_c \\ \cline{2-7}        
    f_l &  -        &  -         &\!{\le0}   &\!{\le0}    &    -    &    -    &\s1    & n_l \\
    f_L & \le0      & \le0       &\!{\le1}   &\!{\^\I{+}1}&  \le0   &  \le0   &\s0    & n_L \\ \cline{2-7}        
    f_r &  -        &  -         &  \le0     &  \le0      &\!{\le0} &\!{\le0} &\s0    & n_r \\
    f_R & \le0      &  \le0      &  -        &  -         &\!{\le0} &\!{\^\I} &\s0    & n_R \\ \cline{2-7}        
    \end{block}  
 \s{d_j}&  \s1      &   \s1      &  \s1      &  \s1       &  \s0    &  \s0    &       &     \\
        &    n_C    &  n_c       &   n_l     &   n_L      &  n_r    &  n_R    &       &
  \end{blockarray}\ .
\end{align}
}%
{\em Notation.} $\I$ denotes an identity matrix of appropriate size (as ``$I$'' is already taken).
$\^\I$ denotes its \sigmx pattern, $0$ on the diagonal and $\ninf$ elsewhere.
So $\^\I+1$ (elementwise addition) is $1$ on the diagonal and $\ninf$ elsewhere.

Three key properties are as follows.
Though clear on inspecting \rf{f_Tf_N2}, P1 and P2 are not trivial---they come from the zero blocks in $F$, see \rf{blockF1}, hence from the tree being optimal (A6).
\begin{enumerate}[P1.]
  \item The only $1$'s occur in the $cC$, $cc$, $lL$ and $LL$ blocks.
  \item All entries of rows $f_C$ [resp.\  $f_l$] are $\ninf$ outside the $CC$ and $Cc$ [resp.\ $ll$ and $lL$] blocks.
  \item The $cc$ and $LL$ blocks both hold $\^\I+1$.
\end{enumerate}
\medskip

From \rf{Sigma}, if the provisional offsets are correct, the number of \dof is
\begin{align}\label{eq:DOFfmla}
  \DOF = \sum_j d_j - \sum_i c_i = (n_C+n_c+n_l+n_L) - (n_C+n_l) = n_c+n_L.
\end{align}
%This must also equal the value $\val\calT = \sum_{(i,j)\in\calT} \sij$ of an HVT $\calT$.

\begin{example}
Below is shown the $\Sigma$ for the example, with the notation for the state vector used in \rf{vecptn2,videf2}.
The voltage and current source rows and columns have not been removed, to show their status as output variables (a $0$ in the $(v_2,v_2)$ and $(I_8,I_8)$ positions and nothing else in the column).
Alongside it are the actual equations, in the same notation.
{\rnc\![1]{\colorbox{lightgray}{$#1$}}
\begin{align}\label{eq:Sigma2}
  \Sigma &=
  \raisebox{-1.2ex}
  {$\begin{blockarray}{rccccccccl}
           & q_{C6}    & q_{c3}     & \phi_{l4} & \phi_{L7}  & \v_{r1} & \v_{R5} & \i_{v2} & \v_{I8} &\s{c_i}&     \\
  \begin{block}{c[cc|cc|cc|cc]l}
    f_{C6} &\!{ 0  }   &\!{ 0 }     &  -        &  -         &    -    &    -    &    -    &    -    &\s1 \\
    f_{c3} &\!{ 1  }   &\!{ 1 }     &  -        &  0         &    -    &    -    &    -    &    -    &\s0 \\ \cline{2-9}        
    f_{l4} &  -        &  -         &\!{ 0  }   &\!{ 0  }    &    -    &    -    &    -    &    -    &\s1 \\
    f_{L7} &  -        &  0         &\!{ 1  }   &\!{ 1  }    &    -    &    -    &    -    &    -    &\s0 \\ \cline{2-9}        
    f_{r1} &  -        &  -         &   -       &   -        &\!{ 0  } &\!{ 0  } &    -    &    -    &\s0 \\
    f_{R5} &  -        &  -         &   -       &   -        &\!{ 0  } &\!{ 0  } &    -    &    -    &\s0 \\ \cline{2-9}        
    f_{v2} &  1        &  -         &   -       &   -        &    -    &    0    &  \!0    &    -    &\s0 \\
    f_{I8} &  -        &  -         &   1       &   -        &    0    &    -    &    -    &  \!0    &\s0 \\ \cline{2-9}
  \end{block}  
    \s{d_j}&  \s1      &   \s1      &  \s1      &  \s1       &  \s0    &  \s0    &  \s0    &  \s0    &       &     %\\
%           &    n_C    &  n_c       &   n_l     &   n_L      &  n_r    &  n_R    &  n_r    &  n_R    &       &
  \end{blockarray}
  $},\quad
  {\rnc\arraystretch{1.13}
  \left\{\begin{array}{r@{\;=}c@{}c}
   -q_{C6}/C_6    & -q_{c3}/c_3    &+ v_2(t) \\
    \dot{q}_{c3}  &  \phi_{L7}/L_7 &- \dot{q}_{C6} \\
    \phi_{l4}/l_4 & -\phi_{L7}/L_7 &+ I_8(t) \\
   -\dot\phi_{L7} &  q_{c3}/c_3    &- \dot\phi_{l4} \\
    \v_{r1}/r_1   & -I_8(t)        &+ \v_{R5}/R_5 \\
   -\v_{R5}       & -v_2(t)        &+ \v_{r1} \\
    \i_{v2}       & \dot{q}_{C6}   &- \v_{R5}/R_5 \\
    \v_{I8}       & -\dot\phi_{l4} &+ \v_{r1}
  \end{array}\right.
  }
\end{align}
}%
Comparing with \fgref{CpH8by5ex} one can verify that each equation is KCL applied to a cutset, or KVL applied to a loop, of the tree $T$.
\end{example}

\subsection{Supporting lemmas}

The proof of \thref{main} uses several lemmas.
\begin{lemma}\label{lm:diagtransversal}
  No transversal $\calT$ of $\Sigma$ can have a value greater than that in \rf{DOFfmla}, and any $\calT$ that achieves this value lies within the diagonal blocks (shaded in \rf{Sigma}).
\end{lemma}
\begin{proof}
  All $\sij\le0$ except for the $n_c{+}n_L$ rows $f_c$ and $f_L$, each having at least one $1$. So $\calT$ cannot do better than use a $1$ in each of these rows, hence $\val\calT \le n_c+n_L$, proving the first half.
  
  For the second half, to give $\calT$ the needed number of $1$'s, all its $n_c$ entries in the $f_c$ rows must be in columns $x_C$ or $x_c$.
  Its $n_C$ entries in the $f_C$ rows are in the same columns because the rest of these rows is $\ninf$.
  Similarly all $\calT$'s entries in the $f_L$ and $f_l$ rows are in columns $x_l$ or $x_L$.
  
  This uses up all the $x_C,x_c,x_l,x_L$ columns and forces $\calT$'s entries in the $f_r,f_R$ rows into the $x_r,x_R$ columns, since there is nowhere else left.
  Hence any $\calT$ with value $n_c+n_L$ lies within the diagonal blocks as asserted.
\end{proof}

We need a result about SPD matrices, similar to Lemma 20 in LS.
\begin{lemma}\label{lm:M_SPD1}
  If $M = \smallbmx{M_{11} &M_{12} \\ M_{21} &M_{22}}
  \smallmx{n_1\\ n_2}$ is SPD then
  \begin{align}
    P &= \raisebox{-2.5ex}{$ %HORRID
  \begin{blockarray}{ccl}
  \begin{block}{[cc]l}
  M_{11} - N^* M_{21}, &M_{12} - N^* M_{22} &\s{n_1} \\
  N, &\I &\s{n_2} \\[-.5ex]
  \end{block}
  \s{n_1} & \s{n_2}
  \end{blockarray}
  $}
  \end{align}\vskip-3ex%HORRID space fix
  \noindent is nonsingular for an arbitrary $n_2\x n_1$ matrix $N$.
\end{lemma}
\begin{proof}
Given $w = \bmx{u\\v}\ \mx{n_1\\ n_2}$ we need to show $Pw=0$ implies $w=0$.
So suppose $\bmx{x\\y} = Pw = 0$. In particular $0=y=Nu+v$ so $v = -Nu$.
Substituting this into the expression for $x$ gives
\[
  0 = x = Y^* \bmx{M_{11} &M_{12} \\ M_{21} &M_{22}} Yu = Y^* M Yu
    \qquad\text{where } Y = \bmx{\I \\ -N} \ \mx{n_1\\ n_2}.
\]
Premultiply by $u^*$, then $(Yu)^* M (Yu) = 0$.
As $M$ is SPD, $0 = Yu = \bmx{u\\-Nu} = \bmx{u\\v} = w$.
\end{proof}

In preparation for the next result, write the SPD matrices of Assumptions (A3)--(A5) in block form matching \rf{Sigma} as follows.
(As in \rf{Sigma}, for capacitors, cotree items come first; for inductors and resistors, tree items do.)
\begin{align}\label{eq:CLGmatrix1}
   \?C &= \bmx{\?C_{CC} &\?C_{Cc} \\ \?C_{cC} &\?C_{cc}}\ \mx{n_C\\ n_c}\,,
  &\?L &= \bmx{\?L_{ll} &\?L_{lL} \\ \?L_{Ll} &\?L_{LL}}\ \mx{n_l\\ n_L}\,,
  &\?G &= \bmx{\?G_{rr} &\?G_{rR} \\ \?G_{Rr} &\?G_{RR}}\ \mx{n_r\\ n_R}\,.
\end{align}

\begin{lemma}\label{lm:nonsingJ}
  The \sysJ $\_J$ defined by the provisional offsets is nonsingular.
\end{lemma}

\begin{proof}
$\_J$ is defined by
\begin{align}\label{eq:sysJdef}
\_J_{ij} = 
\begin{cases}
  \dbdt{f_i}{x_j^{(d_j-c_i)}} &\text{where $d_j-c_i\ge0$}, \\
   0 &\text{elsewhere}.
\end{cases}
\end{align}
We prove nonsingularity by showing $\det(\_J)\ne0$.
Now $\det(\_J)$ is a sum of ($\pm$) products of matrix entries on a transversal, taken over all transversals of $\_J$---where the latter means by definition a transversal $\calT$ such that $\_J_{ij}\ne0$ for all $(i,j)\in\calT$.

For any such $\calT$, the definition of \sigmx and \rf{sysJdef} imply that $d_j-c_i = \sij$ holds everywhere on $\calT$, hence 
$\val\calT =  \sum_{(i,j)\in\calT} \sij = \sum_j d_j - \sum_i c_i = n_c+n_L$.
Hence, by \lmref{diagtransversal}, all transversals of $\_J$ lie within the shaded diagonal blocks in \rf{Sigma}.

Hence $\det(\_J)$ is the product of the determinants of the corresponding diagonal blocks of $\_J$, call them $\_J_{\?C}$, $\_J_{\?L}$ and $\_J_{\?G}$.
(There are usually many nonzeros of $\_J$ outside the diagonal blocks, but this shows they cannot contribute to $\det(\_J)$.)

\para{The $\_J_{\?C}$ block}
Expanding the relevant components of \rf{f_Tf_N2} gives
\begin{align}
  f_C &= \dbd{H}{x_C}(x_c,x_C) &+& 
         \left( \$Cc \dbd{H}{x_c}(x_c,x_C) + \cancel{\$Cv V(t)} \right) ,\\
  f_c &= \xp_c &-& 
         \left(\cancel{\$Lc^* \dbd{H}{x_L}(x_l,x_L)} + \$Cc^* \xp_C +
         \cancel{\$Lc^* I(t)} + \cancel{\$Rc^* g_N(x_r,x_R)} \right),
\end{align}
where terms with no dependence on $(x_C,x_c)$ are struck out.
Now $d_j-c_i$ equals $0$ in the $f_C$ rows and $1$ in the $f_c$ rows so by \rf{sysJdef}
\begin{align}\label{eq:JCblock}
  \_J_{\?C} &= \bmx{\ds\dbd{f_C}{x_C}& \ds\dbd{f_C}{x_c}\\ \ds\dbd{f_c}{\xp_C}& \ds\dbd{f_c}{\xp_c}} 
   = \bmx{\ds\dbd{^2 H}{x_C^2} + \$Cc \ddbd{H}{x_C}{x_c}, & \ds\ddbd{H}{x_c}{x_C} + \$Cc \dbd{^2 H}{x_c^2} \\[3ex]
  -\$Cc^*, & \I} \\
  &= \bmx{\?C_{CC} + \$Cc \?C_{cC}, & \?C_{Cc} + \$Cc \?C_{cc} \\[1ex]
  -\$Cc^*, & \I}, \qquad\text{with $\?C$ from \rf{CLGmatrix1}}.
\end{align}
Setting $M=\?C$, which is SPD by (A3), and $N=-\$Cc^*$ in \lmref{M_SPD1} shows $\_J_{\?C}$ is nonsingular.

\para{The $\_J_{\?L}$ block}
Working as in the previous paragraph we find
\begin{align}\label{eq:JLblock}
  \_J_{\?L} &= \bmx{\ds\dbd{f_l}{x_l}& \ds\dbd{f_l}{x_L}\\ \ds\dbd{f_L}{\xp_l}& \ds\dbd{f_L}{\xp_L}}
%   \\
%  &= \bmx{\ds\dbd{^2 H}{x_l^2} - \$lL \ddbd{H}{x_l}{x_L}, & \ds\ddbd{H}{x_L}{x_l} - \$lL \dbd{^2 H}{x_L^2} \\[3ex]
%  \$Ll, & \I} 
  = \bmx{\?L_{ll} - \$Ll^* \?L_{Ll}, & \?L_{lL} - \$Ll^* \?L_{LL} \\[1ex]
  \$Ll, & \I}.
\end{align}
Setting $M=\?L$, which is SPD by (A4), and $N=\$Ll$ in \lmref{M_SPD1} shows $\_J_{\?L}$ is nonsingular.

\para{The $\_J_{\?G}$ block}
Working as before we find
\begin{align}\label{eq:JGblock}
  \_J_{\?G} &= \bmx{\ds\dbd{f_r}{x_r}& \ds\dbd{f_r}{x_R}\\ \ds\dbd{f_R}{\xp_r}& \ds\dbd{f_R}{\xp_R}}
%   \\
%  &= \bmx{\ds\dbd{^2 H}{x_r^2} - \$Rr^* \ddbd{H}{x_r}{x_R}, & \ds\ddbd{H}{x_R}{x_r} - \$Rr^* \dbd{^2 H}{x_R^2} \\[3ex] \$Rr, & \I}
  = \bmx{\?G_{rr} - \$Rr^* \?G_{Rr}, & \?G_{rR} - \$Rr^* \?G_{RR} \\[1ex]
  \$Rr, & \I}.
\end{align}
Setting $M=\?G$, which is SPD by (A5), and $N=\$Rr$ in \lmref{M_SPD1} shows $\_J_{\?G}$ is nonsingular.

We saw above that $\det(\_J) = \det(\_J_{\?C}) \det(\_J_{\?L}) \det(\_J_{\?G})$.
We have now shown this is nonzero, completing the proof.
\end{proof}

Up to now it might even have been possible that the DAE is structurally ill-posed, i.e.\ that $\Sigma$ has no finite transversal.
That this cannot happen is due to assumptions (A3)--(A5), which were deployed in the previous lemma.
\begin{lemma}\label{lm:cidj_correct}
  The DAE is structurally well-posed; $c_i$, $d_j$ shown in \rf{Sigma} are its canonical offsets.
\end{lemma}
\begin{proof}
From \lmref{nonsingJ} at least one transversal $\calT$ of $\_J$ exists, which is therefore a finite transversal of $\Sigma$.
By definition this means the DAE is structurally well-posed.
By \lmref{nonsingJ} $\calT$ has value $n_c+n_L$ so by \lmref{diagtransversal} it lies within the diagonal blocks.
Since the only $1$'s of $\Sigma$ in these blocks are in the $f_c$ and $f_L$ rows, every $(i,j)$ of $\calT$ in these rows must have $\sij=1$.
All other $(i,j)$ of $\calT$ have $\sij=0$.

The canonical offsets are defined to be the smallest numbers $\ge0$ satisfying
\begin{align}\label{eq:cidjtest}
  \text{$d_j-c_i\ge\sij$ everywhere with equality on some transversal of $\Sigma$}.
\end{align}
They may be computed by the ``$\_c,\_d$ iteration'', \cite[Algorithm 3.1]{Pryce2001a}.
Briefly this says: start with all $c_i=0$. (a) Set the $d_j$ to the smallest values that make $d_j-c_i\ge\sij$ everywhere.
(b) Set $c_i$ to the unique values that make $d_j-c_i=\sij$ on $\calT$. Repeat from (a) until nothing changes.
 It is easily seen that, applied here, this terminates after one iteration with the indicated offsets.
\end{proof}

\subsection{Main result}

\begin{proof}[Proof of \thref{main}]
  It is optional whether we include the output variables $x_v,x_I$ as unknowns in the DAE; they do not affect whether it is SA-amenable.
  For the version that omits them, and has \sigmx \rf{Sigma}, the above lemmas have shown (a) the Jacobian defined using certain offsets is nonsingular, (b) these offsets are the canonical ones.
  Hence the DAE is SA-amenable.
\end{proof}

\section{Consequences}

\subsection{Index}

\begin{corollary}\label{co:indexformula}
  Let $A$ mean ``all capacitors are in the optimal tree and all inductors outside it'' and $B$ mean ``there are no resistors''.
  Then the differentiation index $\nu_d$ is $0$ if both $A$ and $B$ are true, $1$ if just one of them is false, and $2$ if both are false.
\end{corollary}
\begin{proof}
  For an SA-amenable DAE, $\nu_d$ is the largest $c_i$, plus $1$ if any $d_j$ is zero.
  $A$ is equivalent to $n_C=n_l=0$, hence to $\max_i c_i=0$ because there are no $f_C$ or $f_l$ equations; $B$ is equivalent to $n_r=n_R=0$, hence to no $d_j$ being zero.
  The result follows at once.
\end{proof}
This accords with the Tischendorf index theorem \cite{Tischendorf99topologicalindex} for the MNA model.

\subsection{Numerical solution}

\subsubsection{By dummy derivatives and a DASSL-style solver}

DAE index reduction consists in appending to the system the $t$-derivatives of some equations to get a larger system that can be solved to give a system of lower index.

Dummy Derivatives (DDs) \cite{Matt93a} is a structural/numerical way to choose some derivatives $x_j^{(q)}$ of the variables $x_j$ to treat as {\em algebraic} unknowns, removing the seeming over-determinedness in the above recipe.
\cite{Pryce2015DDsnew} gives a \Smethod perspective on the process.
The choice depends on Jacobian condition numbers, so in general may change (dummy pivoting) during numerical solution.
In simple cases such as constant-coefficient linear it may be made once for all beforehand.

We summarise DDs for the index 2 case here, so $n_C+n_l>0$ in \rf{Sigma}---some $f_C$ or some $f_l$ equations exist.
They have $c_i=1$ so differentiate each one once.
By the block structure in \rf{Sigma} and hence in $\_J$ we can handle the C and the L elements separately.

As $\_J_{\?C}$ in \rf{JCblock} is nonsingular its upper part, rows $f_C$, columns $x_C,x_c$ of \rf{Sigma}, has full row rank $n_C$.
Choose $n_C$ linearly independent columns of this part, which might come from $x_C$ or $x_c$ columns or both.
They form an $n_C{\x}n_C$ nonsingular matrix $K_{\?C}$. Let the corresponding variables form vector $y_{\?C}$ and the remaining $n_c$ variables in these columns form vector $s_{\?C}$.
By the Implicit Function Theorem we can solve $f_C=0$ for $y_{\?C}$ as a function of $s_{\?C}$, say $y_{\?C} = \psi_{\?C}(s_{\?C})$.

Similarly, picking out an $n_l{\x}n_l$ nonsingular part $K_{\?L}$ of the upper part of Jacobian $\_J_{\?L}$, we solve for $n_l$ of the $(x_l,x_L)$ variables as a function of the remaining $n_L$, say $y_{\?L} = \psi_{\?L}(s_{\?L})$.

Then the $(n_C+n_l)$ variables $s = [s_{\?C};\, s_{\?L}]$ are the {DDs state vector}, such that the whole DAE is equivalent to a first-order implicit ODE $F(t,s,\dot{s})=0$.
See \cite{Pryce2015DDsnew} for how each $F$-evaluation, in the linear case, entails one linear solve with each of $K_{\?C}$, $K_{\?L}$ and the whole \sysJ $\_J$.

\subsubsection{By \daets}

The \daets high-index DAE code \cite{nedialkov2008solving}, written in \cpp, solves this kind of DAE directly.
Being based on automatic differentiation it accepts the DAE in the form \rf{CpH1,videf1} with little change.
It has solved many small circuits in CpH form but we have not so far tested its performance on large circuits.

\subsubsection{By reduction to explicit ODE}

Using DDs it is easy to go beyond the implicit ODE form to cast the system as an explicit ODE $\dot{s} = f(t,s)$.
It is explained in \cite{Pryce2015DDsnew} why this might be desirable, e.g.\ for storage reasons, when the number of DOF is small compared with the total number of variables.
For instance the ESI/CyDesign simulation software, mainly used for mechanical systems, takes this approach.
The NAG Library in 2015 provided a ``reverse-communication'' Runge--Kutta solver for precisely such applications.

\section{Conclusions}

CpH is a simple method of formulating a DAE for an RLC circuit.
Like that of the Branch-Oriented Method and unlike that of Modified Nodal Analysis the DAE is always SA-amenable.

It is smaller than either thanks to following \pH practice, namely for charge in a capacitor and flux linkage in an inductor to be independent rather than dependent variables.
With $m$ edges and $n$ nodes in the circuit graph, the DAE has size around $2m$ for BOM, around $m+n$ for MNA, and around $m$ for CpH.

Our proof relies on the same well-posedness assumptions as those of Scholz in her 2017 study of \SA applied to various circuit models.
Our arguments draw heavily on her BOM analysis but are simpler because of the smaller size of the system.

We have simulated a large number of small circuits and confirmed using the \daets solver that it is easy to convert a circuit specification automatically to code using the CpH model.

\appendix
\small
\section{\cpp code for the running example}\label{sc:excode}

This presents the DAE of the example in a form accepted by the Nedialkov--Pryce solver \daets.

\para{Notes}
\begin{enumerate}[1.]
  \item The function is templated on a type \li{T}, which \daets instantiates several times with discrete types to perform structural analysis, and then with a numeric type to do the integration.
  \item In lines 4--11, the circuit elements have been given values by the code generation process, to create a self-contained piece of code that can be run by \daets.
  Numeric items are given pseudo-random values in $(0,1)$; voltage and current sources are set to an anonymous function chosen at random from a small gallery of functions of $t$. The user can adjust these.
  \item The state vector $x$ is as in \rf{exstatevar}.
  Lines 14--21 set the vectors $\v,\i$ in terms of it, again ordered by edge numbers.
  \item Lines 23--30 apply the Dirac structure, to form the residuals of the DAE.
  Again these are done in order of edge number---it just happens for this example that tree edges (KCL) come first followed by cotree edges (KVL).
\end{enumerate}
\begin{lstlisting}[language=C++,basicstyle=\footnotesize]
template <typename T>
void fcn(T t, const T *x, T *f, void *param) {
  // Function to specify circuit P8by5 for DAETS
  const double R1 = 0.8666;
  const auto V2 = [](T t) -> T {return cos(t);};
  const double C3 = 0.50689;
  const double L4 = 0.91901;
  const double R5 = 0.58256;
  const double C6 = 0.48617;
  const double L7 = 0.57219;
  const auto I8 = [](T t) -> T {return 2*sin(3*t);};
  // Port variables
  T v[8], i[8];
  v[0] = x[0];            i[0] = x[0]/R1;
  v[1] = V2(t);           i[1] = x[1];
  v[2] = x[2]/C3;         i[2] = Diff(x[2],1);
  v[3] = Diff(x[3],1);    i[3] = x[3]/L4;
  v[4] = x[4];            i[4] = x[4]/R5;
  v[5] = x[5]/C6;         i[5] = Diff(x[5],1);
  v[6] = Diff(x[6],1);    i[6] = x[6]/L7;
  v[7] = x[7];            i[7] = I8(t);
  // Dirac structure, CpH model, tree={0,1,2,3}, cotree={4,5,6,7}
  f[0] = -i[0]+i[4]-i[7];
  f[1] = -i[1]-i[4]+i[5];
  f[2] = -i[2]-i[5]+i[6];
  f[3] = -i[3]-i[6]+i[7];
  f[4] =  v[4]+v[0]-v[1];
  f[5] =  v[5]+v[1]-v[2];
  f[6] =  v[6]+v[2]-v[3];
  f[7] =  v[7]-v[0]+v[3];
}
\end{lstlisting}

%%%%%%%%%%%%%%%%%%%%%%%%%%%%%%%%%%%%%%%%%%%%%%%%%
%\bibliographystyle{plain}
%\bibliography{../../Bibliography/NedBib}

\end{document}